\documentclass[a4paper,11pt,twoside]{amsart}

\usepackage[english]{babel}
\usepackage[latin1]{inputenc}
\usepackage[T1]{fontenc}
\usepackage{amsfonts}
\usepackage{amsmath}
\usepackage{amsthm}
\usepackage{amscd}
\usepackage{float}
\usepackage{enumerate}
\usepackage{multirow}
\usepackage{graphicx}
\usepackage{latexsym}
\usepackage{amssymb}
\usepackage[all]{xy}
\usepackage{mathrsfs}
\usepackage{url}
\usepackage{xstring}
\usepackage{ifthen}
\usepackage{listings}
\usepackage{color}

\lstdefinelanguage{Magma}
{
keywords={if,for,end,then,else,elif,while,function,return,cat,&,and,or },
morekeywords={HurwitzAction,Seqset,Setseq,Polytope,AutomorphismGroup,RowSequence,IdentifyGroup,
	      Subgroups,PermutationMatrix,Generators,MatrixGroup,Transpose},
sensitive=false,
morecomment=[l]{//},
morecomment=[s]{/*}{*/},
morestring=[b]",
}

\lstnewenvironment{codice_magma}[1][]
{\lstset{basicstyle=\tiny\ttfamily, columns=fullflexible,
language=Magma,
keywordstyle=\color{red}\bfseries,
commentstyle=\color{blue},
numbers=none, numberstyle=\tiny,
stepnumber=2, numbersep=5pt, frame=single, #1}}{}

\theoremstyle{plain}
\newtheorem{teor}{Theorem}[section]

\newtheorem{prop}[teor]{Proposition}
\newtheorem{lemma}[teor]{Lemma}
\newtheorem{defn}[teor]{Definition}

\newtheorem{que}[teor]{Problem}

\theoremstyle{definition}
\newtheorem{oss}[teor]{Remark}
\newtheorem{ese}[teor]{Example}

\newcommand{\Ni}{\operatorname{Ni}}
\newcommand{\id}{1}
\newcommand{\aut}{\operatorname{Aut}}
\newcommand{\cc}{{C_{2 \times 2}}}
\renewcommand{\H}{\mathcal{H}}
\renewcommand{\P}{\mathbb{P}}

\newcommand\conn[1]{\raisebox{.5pt}{\textcircled{\raisebox{-.9pt}{#1}}}}

\renewcommand{\a}[1]{
(\StrChar{#1}{1}, \StrChar{#1}{2})(\StrChar{#1}{3}, \StrChar{#1}{4})
}

\renewcommand{\b}[1]{
(\StrChar{#1}{1}, \StrChar{#1}{2})
}

\begin{document}
\title{Hurwitz spaces and liftings to the Valentiner group}
\author{Riccardo Moschetti and Gian Pietro Pirola}

\address{R. Moschetti\\Department of Mathematics and Natural Sciences, University of Stavanger, \\NO-4036 Stavanger, Norway}
	\email{riccardo.moschetti@uis.no}

\address{G. Pirola\\Dipartimento di Matematica,
              Universit\`a di Pavia\\Italy}
\email{gianpietro.pirola@unipv.it}

\thanks{}
\subjclass[2010]{14D05, 14H30, 14Q05}
\keywords{Monodromy, Hurwitz spaces, Valentiner group.}
\begin{abstract}
We study the components of the Hurwitz scheme of ramified coverings of $\P^1$ with monodromy given by the alternating group $A_6$ and elements in the conjugacy class of product of two disjoint cycles. In order to detect the connected components of the Hurwitz scheme, inspired by the case of the spin structures studied by Fried for the $3$-cycles, we use as invariant the lifting to the Valentiner group, triple covering of $A_6$. We prove that the Hurwitz scheme has two irreducible components when the genus of the covering is greater than zero, in accordance with the asymptotic solution found by Bogomolov and Kulikov.
\end{abstract}

\maketitle

\section{Introduction}
The symmetric and the alternating groups are ubiquitous in the study of the monodromy of curves. In particular, they are the unique possible examples of monodromy if we are in the case of an indecomposable cover $X \rightarrow \P^1$ with $X$ a generic complex curve of genus greater than $3$ (see for instance \cite{GS} and \cite{GM}). The proof of the existence of such a covering for the general curve and symmetric monodromy is classical and, for alternating monodromy, can be found in \cite{MV} and \cite{AP}. 
Coverings with odd ramification have been studied starting from the seminal works of Serre in \cite{SerreRelevament}, and Fried in \cite{FriedAlternating}. This is particularly interesting due to the relations with theta characteristics and modular towers studied in \cite{BF}. 
Serre proved in \cite{SerreRelevament} that the moduli spaces of absolute and inner covers of $\P^1$ of genus zero with monodromy group $A_n$ and elements in conjugacy classes of odd order cycles are connected. For a good reference about alternating groups in which one can find all the precise definition of conjugacy classes and coverings see \cite{WG}. It is proved in \cite{FriedAlternating} that the spin structure determines the irreducible components of the Hurwitz space of coverings of $\P^1$ of degree $n$ branched on $r$ points with $r \geq n \geq 5$, and monodromy given by the conjugacy class of $3$-cycles in $A_n$. In both cases, the crucial point consists on the construction of the so-called lifting invariant, described algebraically by using a lifting to the double cover of $A_n$. This strategy does not work if one considers elements with even order. The cases of alternating groups of order six and seven are exceptional because they are the only two examples of alternating groups that admit coverings of degree three and six. This makes it possible to construct lifting invariants in case of other ramification type, for instance of order two. Bogomolov and Kulikov prove in \cite{BK} that the number of the connected components of the Hurwitz scheme is asymptotically determined by the ambiguity index, defined in terms of group coverings. The cases of $A_6$ and $A_7$ turn out to be different, as expected, with respect to the other alternating groups. This makes interesting to study the irreducible components of the Hurwitz schemes for lower genus for these cases. 

In this work, we study the irreducible components of the Hurwitz scheme in the case of the conjugacy class of products of two disjoint cycles of $A_6$. The strategy consists of an induction on the number of the branch points of the covering, using the lifting invariant to distinguish the different components. Our main result is Theorem \ref{teor:finalresult}, which asserts that the Hurwitz scheme has two irreducible components when the genus is greater than zero, and only one component for genus zero.
We also study in Theorem \ref{teor:finalresultinner} the case of the inner moduli space, proving that there are three irreducible components for genus greater than zero and two for genus zero. This provides an explicit example in which the space of the Galois closures of coverings of $\P^1$ has more connected components than the Hurwitz scheme. As for the $3$-cycles, treated in \cite{FriedAlternating}, we get that the minimal bound on the genus $g$ of $X$ from which the results of \cite{BK} hold is $g > 0$, for the case of the product of two disjoint cycles of $A_6$. 

\textbf{The plan of the paper.} Some preliminaries concerning coverings, monodromy and Hurwitz spaces are carried on in Section \ref{sec:prel}. Theorem \ref{teor:finalresult}, concerning the absolute moduli space, is mainly proved in Section \ref{sec_absolute}. The two base cases of the induction, namely the case of five and six points, are carried on in Section \ref{App:MonGenZero} and Section \ref{App:MonGenOne}, respectively. The proof of Theorem \ref{teor:finalresultinner}, that is the main result in the case of inner moduli spaces, carried on in Section \ref{sec_inner}, uses the results of the previous sections. Some open problems and some ideas for further work are stated in Section \ref{sec:open}. Finally, Appendix \ref{Apx:code} contains all the MAGMA codes used in the proofs.

\section{Preliminaries} \label{sec:prel}
In this paper all the varieties will be defined over the complex numbers. Let $X \xrightarrow{f} \P^1$ be a covering of the sphere of degree $n$. From now on, $k$ will always denote the number of ramification points of such a covering. A branch point for $f$ is a $\bar{z} \in \P^1$ such that the fibre over $\bar{z}$ is composed by a number of points that is strictly less than $n$. Let $Z$ be the branch locus of $f$; the fundamental group $\pi_1(\P^1 \smallsetminus Z, z_0)$ is generated by laces $[\gamma_1], \ldots, [\gamma_r]$ around the branch points modulo the relation $\prod [\gamma_i]=\id$. Denote by $F_0$ the fibre $f^{-1}(z_0)$, $z_0 \notin Z$.
After the choice of $p_0 \in f^{-1}(z_0)$, by the unicity of the lifting, every element $[\gamma_i]$ lifts in a unique way to $\bar{\gamma}_i \in \pi_1(X \smallsetminus f^{-1}(Z), p_0)$ as a path starting in $p_0$. This gives a well defined map
\begin{align*}
m:\pi_1(\P^1 \smallsetminus Z, z_0) &\rightarrow \aut(F_0)\\
[\gamma_i] &\mapsto (m([\gamma_i]): p_j \mapsto \bar{\gamma}_i(1))
\end{align*}
The image of $m$ is called the \textbf{monodromy group} of $f$ in $p_0$. If a name of the points of the fibre $F_0$ is given, one can associate an automorphism $m([\gamma_i])$ to an element of the symmetric group in $n$ elements.
From the geometric point of view, one can consider the Hurwitz space of coverings $\{X \rightarrow \P^1\}$ with fixed ramification type. There are several equivalence relations on this space. 
In a natural way, two coverings $\{X \xrightarrow{f} \P^1\}$ and $\{Y \xrightarrow{g} \P^1\}$ are considered equivalent if there exists a biholomorphism $\psi:X \rightarrow Y$ such that $f=g \circ \psi$. It is easy to prove that this equivalence relation is the strongest that fixes the branch points, the ramification type, and, up to conjugation, the monodromy group $G$ of the cover.
This defines the \textbf{Hurwitz spaces} $\H(G,C)^{abs}$ or $\H([G],C)^{abs}$, depending if one want to fix the monodromy $G$ or just to consider it up to conjugation in $S_n$. 
Another useful equivalence relation is obtained by considering the Galois closure of such coverings. In this case one has to choose a connected component of the $n$-fold fibre product of $f$, and the \textbf{inner Hurwitz space} $\H(G,C)^{in}$ is defined by considering all such choices to be equivalent. 
An algebraic description of these spaces can be given by using Nielsen classes. 
An introduction to the theory of these classes can be found in \cite{FriedOnline} and \cite{MS} contains some background on Hurwitz spaces.
Let $G$ be a transitive subgroup of $S_n$ and consider $r$ conjugacy classes $C:=(c_i)_{i=1}^r$ of $G$. An element $\pmb{g}$ in the \textbf{Nielsen class} $\Ni(G,C)$ is given by $r$ elements $g_i$ of $G$ such that $\prod g_i = \id$, $g_i \in c_i$ and the subgroup $\langle g_i \rangle$ generated by the $g_i$ is $G$. The \textbf{Riemann's existence theorem} guarantees that given a covering $X \xrightarrow{f} \P^1$, then the $r$ elements of $S_n$ associated to the images $m([\gamma_i])$ belong to a certain Nielsen class. Conversely, given an element $\pmb{g}$ in a Nielsen class, there exists a covering $X \xrightarrow{f} \P^1$ that is associated to $\pmb{g}$. Some background on Riemann's existence theorem can be found in \cite{FEnhanced}, \cite{FTwist}, \cite{MM}, \cite{SGal} and \cite{V}. 
There are two group actions that can be defined on $\Ni(G,C)$: choosing a different name for the elements of the fibre $F_0$ would give a right action of $S_n$. In this way only the conjugacy class of $G$ in $S_n$ is fixed and we will denote the Nielsen space as $\Ni([G],C)$. If instead we want to keep the group $G$ fixed, we have to act just by elements of $N_{S_n}(G)$, the normalizer of $G$ in $S_n$. This is called the \textbf{absolute action}. 
The group $G$ also acts on $\Ni(G,C)$ by conjugation, giving the so-called \textbf{inner action}.
If $\pmb{g}:=(g_1, \ldots, g_r)$ is an element in $\Ni(G,C)$ and $s$ is the acting element in $N_{S_n}(G)$ or in $G$, the action, denoted by $\phi_s$ is
$$\phi_s(\pmb{g}) = \phi_s(g_1, \ldots, g_r) = (s^{-1} g_1 s, \ldots, s^{-1} g_r s).$$
Hurwitz space are related with Nielsen spaces via the quotient by these equivalence relations: $\Ni(G,C)/N_{S_n}(G)$ is related to $\H(G,C)^{abs}$, and $\Ni(G,C)/G$ is related to $\H(G,C)^{in}$. As specified before, it is also possible to consider only the conjugacy class of $G$ in $S_n$, obtaining the relation between $\Ni([G],C)/{S_n}$ and $\H([G],C)^{abs}$.

There is another action on these spaces, called the \textbf{Hurwitz action}. The cardinality of the spaces after the quotient by this action is equal to the number of the connected components of the Hurwitz spaces. From the geometric point of view, one can imagine exchanging two branch points by a continued movement, whereas from the algebraic point of view, this action is described by the following

\begin{defn} \label{defn:HAC}
Let $\pmb{g}:=(g_1, \ldots, g_r)$ be an element in $\Ni(A_n,C^r)$. The braid group $B_r$ on $r$ elements acts on the right on this set. By following the notation of \cite{HBC}, consider a generator $\sigma_i$ of $B_r$. The action of $\sigma_i$ on $\pmb{g}$ is given by 
$$\sigma_i(\pmb{g}):=(g_1, \ldots, g_{i+1}, g_{i+1}^{-1} g_i g_{i+1},  \ldots, g_r).$$
\end{defn}
\noindent A description of the Hurwitz action, together with example of monodromies, can be found in \cite{HM} and a point of view on the study of the connected components of the Hurwitz schemes by means of semigroups over groups is carried out in \cite{KK}. Notice that if the group $G$ is the alternating group $A_n$, then $N_{S_n}(A_n)$ is the whole $S_n$, because $A_n$ is unique in its conjugacy class. We will need the following lemma that implies immediately that the inner action on $\Ni(G,C)$ can actually be obtained by using the Hurwitz action only.

\begin{lemma}[Lemma 2.6 of \cite{FriedAlternating}] \label{lemma:friedproductone}
Let $\pmb{g}:=(g_1, \ldots, g_r)$ be an element in $Ni(G,C)$ such that there exists $j<r$ consecutive integers $\{g_i, g_i+1, \ldots, g_j\}$ with $\prod_{h=i}^j g_h =1$. If we denote by $\gamma$ an element in the subgroup generated by $\{g_i, g_i+1, \ldots, g_j\}$, then there exists an element $Q \in B_r$ such that
$$Q(\pmb{g}) = (g_1, \ldots, g_{i-1}, \gamma g_i \gamma^{-1}, \gamma g_{i+1} \gamma^{-1}, \ldots, \gamma g_j \gamma^{-1}, g_{j+1}, \ldots g_r)$$ 
\end{lemma}

The triple covering of $A_6$ is a group of $1080$ elements called the \textbf{Valentiner group}. This group was discovered by Valentiner in \cite{V}, and then studied by Wiman and Gerbaldi in \cite{W} and \cite{Ger}. This covering is described by using the following exact sequence, where $V$ is the Valentiner group and $C_3$ is the cyclic group of order three.

$$0 \rightarrow C_3 \rightarrow V \rightarrow A_6 \rightarrow 0$$

There exist many explicit descriptions of the Valentiner group, together with the covering map to $A_6$, see for instance \cite{At}. In this work we will identify $A_6$ with the subgroup of $S_6$ generated by 
$$\{s_1, s_2\}=\{(1,2)(3,4),(1,2,4,5)(3,6)\}\text{.}$$
The Valentiner group is described as the subgroup of $S_{18}$ generated by
\begin{align*}
\{v_1,v_2\}=\{&(2, 6)(4, 11)(7, 9)(8, 13)(10, 14)(12, 16), \\
&(1, 2, 7, 4)(3, 8, 6, 10)(5, 9, 13, 12)(11, 15)(14, 17)(16, 18)\}\text{.}
\end{align*}
The covering map $V \xrightarrow{\pi} A_6$ is defined on the generators by
$$v_1 \mapsto s_1 \qquad v_2 \mapsto s_2.$$
Let $\cc$ denote the conjugacy class of $A_6$ given by the product of two disjoint cycles. An element $x$ in $\cc$ admits a unique lift $\hat{x}$ to $V$ of order $2$. Let $\pmb{g}:=(g_1,\ldots, g_k)$ be an element in the Nielsen class $\Ni(A_6,\cc^k)$; for all the $g_i$, let $\hat{g}_i$ be the lifting of order $2$ to the Valentiner group. Since the product of the $g_i$ is the identity, then the product of the liftings $\gamma(\pmb{g}):= \prod{\hat{g}_i}$ belongs to the preimage $\pi^{-1}(\id_{A_6})$. 


\begin{prop} \label{prop:InnerAbsolute}
Let $\pmb{g}:=(g_1,\ldots, g_k)$ be an element in $\Ni(A_6,\cc^k)$.
\begin{enumerate}
\item The Hurwitz action commutes with the absolute and inner actions.
\item The element $\gamma(\pmb{g})$ is an invariant of the Hurwitz action.
\item The element $\gamma(\pmb{g})$ is an invariant of the inner action.
\item The order of the element $\gamma(\pmb{g})$ is an invariant of the absolute action.
\end{enumerate}
The element $\gamma(\pmb{g})$ will be called the lifting invariant of $\pmb{g}$.
\end{prop}
\begin{proof}
Even if part (1) can be proved by direct computation, it is interesting to look at it from the point of view of geometry: since the absolute and inner actions are just a choice of the names of the fibres the claim is straightforward.

\noindent To prove (2), consider one of the generators of the braid group $B_r$, $\sigma_1$, which acts on $\pmb{g}$ by $\sigma_1(g_1,\ldots, g_k)=(g_2, g_2^{-1} g_1 g_2, \ldots)$. Then one obtains
$$\gamma(\sigma_1(\pmb{g}))=\hat{g}_2 \cdot \hat{g}_2^{-1} \cdot \hat{g}_1 \cdot \hat{g}_2 \cdot \ldots = \gamma(\pmb{g})\text{.}$$

\noindent To prove (3), let $t$ be an element in $A_6$, and let $\hat{t}$ be a lifting of $t$ in $V$. Then one obtains
$$\gamma(t^{-1} \pmb{g} t) = \gamma(t^{-1} g_1 t, t^{-1} g_2 t, \ldots) = \hat{t}^{-1} \hat{g}_1 \hat{t} \cdot \hat{t}^{-1} \hat{g}_2 \hat{t} \cdot \ldots = \hat{t}^{-1} \gamma(\pmb{g}) \hat{t} = \gamma(\pmb{g})\text{.}$$
The last equality holds because $\gamma(\pmb{g})$ is in the center of $V$.

\noindent Part (4) holds because all the lifting maps in
$$0 \rightarrow C_3 \rightarrow V \rightarrow A_6 \rightarrow 0$$
can be chosen in a natural way, then the order of $\gamma(\pmb{g})$ is well defined and does not change under the absolute action.
\end{proof}

Notice that the number of possible choices for $\gamma(\pmb{g})$ coincides with the ambiguity index $a(A_6,\cc)$ used in \cite{BK}.

\begin{ese}
Code \ref{prog:liftinginvariant} computes the lifting invariant of an element in $\Ni(A_6,\cc^k)$, the strategy being just a direct computation by using the definition. As an example of the fact that only the order of the lifting is an invariant in $\Ni(A_6,\cc^k)^{abs}$ one can take these two elements of $\Ni(A_6,\cc^5)$
$$\{\a{1234},\a{1324},\a{1425},\a{1623},\a{1635}\}\text{,}$$
$$\{\a{1234},\a{1324},\a{1426},\a{1523},\a{1536}\}\text{.}$$
The two lifting invariants are the two liftings of the identity of order three, and one can see immediately that the automorphism $\phi_{\b{56}}$ sends one element to the other.
\end{ese}


The main result in the case of the absolute moduli space is provided by the following theorem, and it is completely analogous to the results of \cite{FriedAlternating}, namely there is only one connected component for genus zero and exactly two connected components for higher genera.

\begin{teor} \label{teor:finalresult}
The spaces $\H(A_6,\cc^k)^{abs}$, for $k$ greater or equal to six have exactly two connected components 
$$\H_+(A_6,\cc^k)^{abs} \text{ and } \H_-(A_6,\cc^k)^{abs}\text{.}$$
The space $\H(A_6,\cc^5)^{abs}$ is connected. That is, the Hurwitz scheme has exactly one irreducible component if the genus is equal to zero and two irreducible components for genus greater than zero.
\end{teor}

The case of the inner moduli space is slightly different from the results of \cite{FriedAlternating}, due to the fact that the lifting invariant has order three. The connected components turn out to be two for genus zero and three for higher genera. 

\begin{teor} \label{teor:finalresultinner}
Fix once for all $\sigma$, a lifting of order three of the identity. The spaces $\H(A_6,\cc^k)^{in}$, for $k$ greater or equal to six have exactly three connected components 
\begin{itemize}
\item $\H_{\conn{0}}(A_6,\cc^k)^{in}$, the elements that lifts to $\sigma^0=\id_V$.
\item $\H_{\conn{1}}(A_6,\cc^k)^{in}$, the elements that lifts to $\sigma^1$.
\item $\H_{\conn{2}}(A_6,\cc^k)^{in}$, the elements that lifts to $\sigma^2$.
\end{itemize}

The space $\H(A_6,\cc^5)^{in}$ has two connected components 
\begin{itemize}
\item $\H_{\conn{1}}(A_6,\cc^5)^{in}$, the elements that lifts to $\sigma^1$.
\item $\H_{\conn{2}}(A_6,\cc^5)^{in}$, the elements that lifts to $\sigma^2$.
\end{itemize}
\end{teor}

For the sake of simplicity we will denote in the same way an element in a Nielsen class and a class of an element in the Hurwitz space. Choosing an element $\pmb{g}$ in $\H_{\conn{1}}(A_6,\cc^k)^{in}$ means choosing an element in $Ni(A_6,\cc^k)$ with lifting invariant $\sigma^1$, that makes the corresponding covering belong to the connected component $\H_{\conn{1}}(A_6,\cc^k)^{in}$.

\section{The space $\H(A_6,\cc^r)^{abs}$} \label{sec_absolute}
The study of the space $\H(A_6,\cc^r)^{abs}$ for all the genera relies on the study of the monodromy group for genus zero and one. Then, it is possible to carry on an induction on the number of branch points in order to conclude the classification. Algebraically this means passing from $k$ to $k-1$ elements of the conjugacy class by multiplying two of them; geometrically, if these elements correspond to two points $P_1$ and $P_2$, that coincides with considering the loop in the fundamental group obtained by composing a loop around $P_1$ and a loop around $P_2$. This gives rise to a subgroup of the fundamental group that describes a monodromy in a fewer number of points. In order to carry on the induction step, two further results are necessary. First, one has to prove that it is possible to reduce the number of points without changing the monodromy group and second, one has to prove that is possible to perform such a reduction by keeping all the elements in the conjugacy class $\cc$.

In order to solve the first issue it is convenient to consider the problem of finding the minimum number of generators contained in a sequence of elements of $A_6$.
\begin{defn}
Let $G$ be a finite group. The max-length of $G$ is a natural number that coincides with the maximum possible length of a chain of subgroups of $G$.
\end{defn}

\noindent The following proposition holds in general for every finite group.
\begin{prop} \label{prop:fivegenerators}
Let $G$ be a finite group and let $l$ be the max-length of $G$. If $S:=\{s_1, \ldots, s_n\}$ generates $G$, and $n \geq l$, then there are $l$ elements of $S$ that are still generators.
\end{prop}
\begin{proof}
The proof is an induction on the cardinality of $S$. 
If $n=l$ the claim is trivially true.
Assume $n>l$ and let the result be true for $n-1$. Consider $S_1 := S \smallsetminus \{s_1\}$. If $\langle S_1 \rangle =\langle S \rangle$, one can use the induction hypothesis on $S_1$. If $\langle S_1 \rangle \subseteq \langle S \rangle$, then one can consider $S_2 := S_1 \smallsetminus \{s_2\}$ and proceed as before, constructing a chain of subgroups of $G$ that has maximum possible length $l$. It means that, at least at the l-th step, $S_{l+1}$ must be equal to $S_l$ and then one can use the induction hypothesis to conclude the proof.
\end{proof}

The max-length of $A_6$ is five, thus it is always possible to find five generators in a set of cardinality $n \geq 5$. Notice that there exists a set of five generators of $A_6$ such that it is not possible to find among them $4$ elements that still generate the whole $A_6$. However, the following proposition shows that is it possible to get a better result if one considers only elements in the conjugacy class $\cc$.

\begin{lemma}
Let $H$ be a subgroup of $G$ of order $h$, and $g$ an element in $G \smallsetminus H$ of order $2$. Then the subgroup $\langle H,g\rangle$ has order $2*h*x$ for a certain natural number $x$.
\end{lemma}

\begin{lemma} \label{lemma:fourgenerators}
Let $S:=\{s_1, \ldots, s_n\}$ generate $A_6$, with $n \geq 4$, and let the $s_i$ belong to the conjugacy class $\cc$. Then there exists a subset of four elements of $S$ that still generate all $A_6$.
\end{lemma}
\begin{proof}
By Proposition \ref{prop:fivegenerators} one can assume $n$ to be equal to five. Then Code \ref{prog:from5to4} shows that the thesis holds. The strategy is just a case by case analysis listing all the possible sets of $5$ generators, which can be chosen to be disjoint and ordered, and then checking that the claim holds.
\end{proof}

\noindent The reduction from $4$ generators to $3$ generators is not as simple as the previous step. In fact there exist sets of $4$ generators in the conjugacy class $\cc$ that can not be reduced to have cardinality three just by taking one out; an example is provided by
\begin{equation} \label{Eqn:exampleNonRed4To3}
\{\a{1234}, \a{1235}, \a{1246}, \a{1324}\}\text{.}
\end{equation}
In order to proceed further, one has to use the Hurwitz action defined in \ref{defn:HAC} to modify the elements.

\begin{prop} \label{Prop:threeGeneartors}
Let $S:=\{s_1, \ldots, s_n\}$ be a ordered set of generators of $A_6$ with cardinality $n \geq 4$, such that all the $s_i$ belong to the conjugacy class $\cc$. Then, up to the Hurwitz action on $S$, it is possible to find three elements that still generate $A_6$.
\end{prop}
\begin{proof}
By Lemma \ref{lemma:fourgenerators} one can assume $n$ equal to four. Code \ref{prog:final4to3} concludes the proof using a case by case analysis.
\end{proof}

The second issue concerns finding two elements $g_1$ and $g_2$ that can be used to reduce the number of branch points. As described before, from the algebraic point of view the monodromy type can be computed just by multiplying the two elements $g_1$ and $g_2$. Since we are working in $\cc$, it is necessary that $g_1 \cdot g_2$ still belongs to $\cc$.

\begin{prop} \label{Prop:FixedPointReduction}
Let $g_1$ and $g_2$ be two elements in the conjugacy class $\cc$ of $A_6$. If $g_1$ and $g_2$ have the same fixed points, then $g_1 g_2$ either belongs to $\cc$, or is the identity.
\end{prop}
\begin{proof}
Up to an external automorphism of even parity, one can assume the first element to be $\a{1234}$ and the second to be either $\a{1234}$ or $\a{1324}$. The claim is then straightforward.
\end{proof}

The best case scenario for the induction would be finding two elements with the same fixed points in a set of cardinality four.

\begin{prop} \label{Prop:reductionfour}
Let $\{g_1, g_2, g_3, g_4\}$ be an ordered set of elements of $A_6$. Then, up to the Hurwitz action and to external automorphisms either it is possible to find two elements with the same fixed points or the set is
$$\{\a{1234}, \a{1235}, \a{1634}, \a{1645}\}\text{.}$$
\end{prop}
\begin{proof}
One can perform a case by case analysis in which the Hurwitz action is used. This is done with Code \ref{prog:find2fixedpoints}. The strategy is a case by case analysis listing all the possible sets of four elements and then using the Hurwitz action to see if there are two elements with the same fixed points. Its output is made by all the sets for which the Hurwitz action does not work. Up to external automorphisms all these sets are equivalent to 
$$\{\a{1234}, \a{1235}, \a{1634}, \a{1645}\}\text{.}$$
\end{proof}

The previous proposition shows that is not always possible to find the expected reduction in a set of cardinality four. As a consequence, the proof for $k=7$ will not be part of the induction step, and is carried on at the beginning of the proof of Theorem \ref{teor:finalresult}. Luckily, the following proposition shows that it is always possible to find two elements with the same fixed points in a set of cardinality five.

\begin{prop} \label{Prop:reductionfive}
Let $S:=\{g_1, \ldots, g_5\}$ be an ordered set of elements of $A_6$. Up to the Hurwitz action it is possible to find two elements with the same fixed points. 
\end{prop}
\begin{proof}
This is done with a case by case analysis carried out with Code \ref{prog:find2fixedpoints5}. 
\end{proof}

Let $\pmb{g}:=\{g_1, \ldots, g_k\}$ be a ordered set of elements in $\cc$, with $k \geq 5$. By Proposition \ref{Prop:reductionfive}, there are two elements with the same fixed points. Up to the Hurwitz action one can assume them to be $g_1$ and $g_2$. By multiplying these elements, as proved in Proposition \ref{Prop:FixedPointReduction}, two cases can arise. If $g_1 = g_2$, that is the product of $g_1$ and $g_2$ is the identity, one can reduce $\pmb{g}$ to $\bar{\pmb{g}}:=\{g_3, \ldots, g_k\}$, an element of length $k-2$; this is called \textbf{2-reduction}. If $g_1 \neq g_2$, one can reduce $\pmb{g}$ to $\bar{\pmb{g}}:=\{g_1 \cdot g_2, g_3, \ldots, g_k\}$, an element of length $k-1$; this is called \textbf{1-reduction}. The following proposition describes the behaviour of the lifting invariant under such reductions.

\begin{prop} \label{prop:liftingandreduction}
Let $\pmb{g}:=\{g_1, \ldots, g_k\}$ be an ordered set of $k \geq 5$ elements in the conjugacy class $\cc$ such that the product of the $g_i$ is equal to the identity. Let $\bar{\pmb{g}}$ be the reduction of $\pmb{g}$. Then, $\pmb{g}$ and $\bar{\pmb{g}}$ have the same lift to the Valentiner group. 
\end{prop}
\begin{proof}
If $\bar{\pmb{g}}$ is a $2$-reduction, $g_1$ was equal to $g_2$, then also the liftings $\hat{g}_1$ and $\hat{g}_2$ are equal. Then $\hat{g}_1 \cdot \hat{g}_2$ is the identity and the lifting invariants of $\pmb{g}$ and $\bar{\pmb{g}}$ are the same.
If $\bar{\pmb{g}}$ is a $1$-reduction, one has to consider the lift of the element $g_1 \cdot g_2$. Since $g_1 \cdot g_2$ is still in $\cc$ it means that the lifting of $g_1 \cdot g_2$ is the product of the lifting of $g_1$ and the lifting of $g_2$, thus the lifting of $\pmb{g}$ and $\bar{\pmb{g}}$ are the same.
\end{proof}

Neither $\pmb{g}$ or $\bar{\pmb{g}}$ is supposed to be transitive in Proposition \ref{prop:liftingandreduction}, because the canonical lift can be defined for every sequence of elements with products one. Every time this reduction is used on an element $\pmb{g}$ in a Nielsen class, one should check that the result $\bar{\pmb{g}}$ is still transitive. In the induction step of the proof of Theorem \ref{teor:finalresult}, this is guaranteed by Proposition \ref{Prop:threeGeneartors} to select three generators, which guarantees the transitivity. Notice that, in order to obtain an element of $\cc$ from the multiplications of two other elements, it is sufficient that the two elements generate a subgroup of order lesser or equal that four. Despite of this more general result, in the proof of Theorem \ref{teor:finalresult}, it is convenient to show that it is possible to consider only reduction of type $1$ in order to apply the following

\begin{prop} \label{prop:liftinghurwitz}
Let $\pmb{g}$ and $\pmb{h}$ belong to $Ni(A_6,\cc^k)$, $k > 5$, with $1$-reduction to $\bar{\pmb{g}}$ and $\bar{\pmb{h}}$. If $\bar{\pmb{g}}$ and $\bar{\pmb{h}}$ are equivalent under the Hurwitz action, then this action can be lifted to obtain an equivalence between $\pmb{g}$ and $\pmb{h}$.
\end{prop}
\begin{proof}
One can always assume that the $1$-reduction takes place between the first two elements of $\pmb{g}$ and $\pmb{h}$. Fix an external automorphism such that the reduction $\pmb{g} \rightarrow \bar{\pmb{g}}$ is the following
$$(\a{1324}, \a{1423}, g_3, \ldots, g_k)\rightarrow (\a{1234}, g_3, \ldots, g_k)$$
By hypothesis, there exists an Hurwitz action on $\bar{\pmb{h}}$ that makes
$$(h_1\cdot h_2, h_3, \ldots, h_k)$$
equal to 
$$(\a{1234}, g_3, \ldots, g_k)$$
This action can be extended to an action on $\pmb{h}$ just by considering the same action on $h_3 \ldots h_k$ and by conjugating always together the elements $h_1$ and $h_2$. This action makes $\pmb{h}$ equal to
$$(\bar{h}_1, \bar{h}_2, g_3, \ldots, g_k)$$
where $\bar{h}_1\bar{h}_2=\a{1234}$. And $\bar{h}_1$ and $\bar{h}_2$ are elements of a $1$-reduction. Up to the Hurwitz action, just on $\bar{h}_1$ and $\bar{h}_2$, it is possible to choose $\bar{h}_1=\a{1324}$ and $\bar{h}_2=\a{1423}$. This gives an equivalence between $\pmb{g}$ and $\pmb{h}$ and concludes the proof.
\end{proof}

Let us recall the base steps of the induction, which are carried out in Sections \ref{App:MonGenZero} and \ref{App:MonGenOne}.

\begin{prop} \label{Prop:FivePoints}
The monodromy arising from a degree $6$ covering of $\P^1$ ramified on $5$ points with ramification in the conjugacy class $\cc$ of $A_6$ generates three possible subgroups of $A_6$. One subgroup of order $24$, denoted by $G_{24}$, that corresponds to the unique, up to conjugation, transitive immersion of $S_4$ inside $A_6$. One of order $60$, denoted by $G_{60}$, that corresponds to the unique, up to conjugation, transitive immersion of $A_5$ inside $A_6$, and the whole $A_6$.
Each space $\H(A_6,\cc^5)^{abs}$, $\H([G_{24}],\cc^5)^{abs}$ and $\H([G_{60}],\cc^5)^{abs}$ is connected.
\end{prop}

\begin{prop} \label{Prop:SixPoints}
The space $\H(A_6,\cc^6)^{abs}$ has exactly two connected components, denoted by 
$$\H_+(A_6,\cc^6)^{abs} \text{ and } \H_-(A_6,\cc^6)^{abs}$$
\end{prop}

The following is the proof of the main theorem for the absolute moduli spaces.

\begin{proof}[Proof of Theorem \ref{teor:finalresult}]

Let us proceed by induction on the number of points $k$. Propositions \ref{Prop:FivePoints} and \ref{Prop:SixPoints} prove the thesis for $k=5$ and $k=6$, respectively. Assume the thesis holds for less than $k$ points and prove the claims for $k$.

It is easy to prove that all the considered components are not empty. This is done explicitly for $k=5$ and $6$ in the related sections. In the general case, take for instance the element 
$$(\a{1234}, g_2, \ldots, g_{k-1} ) \in \H_+(A_6,\cc^{k-1})^{abs}\text{.}$$
By Proposition \ref{prop:liftingandreduction}, the element 
$$(\a{1324},\a{1423}, g_2, \ldots, g_{k-1} )$$
 has the same lifting invariant, hence belongs to $\H_+(A_6,\cc^{k})^{abs}$. Proposition \ref{prop:InnerAbsolute} makes clear that $\H_+(A_6,\cc^{k})^{abs}$ and $\H_-(A_6,\cc^{k})^{abs}$ are two different irreducible components.

It remains to prove that these components are connected, namely if $\pmb{g}$ and $\pmb{g}'$ belong to the same component, that is they have the same lifting invariant, then they are equivalent under the Hurwitz and the absolute actions. The strategy is to use a $1$-reduction in order to obtain a monodromy on $k-1$ points, then use the induction hypothesis and then lift the equivalence using Proposition \ref{prop:liftinghurwitz}.

Let us work on the element $\pmb{g}$. By using Proposition \ref{Prop:threeGeneartors} one can assume the last three elements, that we will denote as $\{h_1, h_2, h_3\}$, to be generators of $A_6$. This ensures that the reductions will have maximal monodromy. Let us now focus on the remaining $k-3$ elements. If $k > 7$, there are enough elements to apply Proposition \ref{Prop:reductionfive} and get a reduction. For $k=7$, the only possibility is to apply Proposition \ref{Prop:reductionfour} to the remaining $4$ elements. One gets that either there still are two elements with the same fixed points, or the element is in the form
$$(\a{1234}, \a{1235}, \a{1634}, \a{1645},h_1, h_2, h_3)$$
Notice that the product of the first four elements is the identity, and then also the product of the last three elements must be the identity, and this is not compatible with the last three elements generating $A_6$. Then a reduction is possible even for the case $k=7$.

Assume that we are facing $2$-reduction, if $k>5$ the element is then in the form
$$\pmb{g}:=(x,x,g_1,\ldots,g_{k-5},h_1, h_2, h_3)$$
In this case, the product $g_1 \ldots h_3$ is the identity, since the $h_i$ are generators there exists an element $\gamma$ in $\langle g_1,\ldots,g_{k-5},h_1, h_2, h_3 \rangle$ such that $x$ and $\gamma g_1 \gamma^{-1}$ makes a $1$-reduction. By Lemma \ref{lemma:friedproductone}, there exists an element $Q \in B_r$ such that
$$Q(\pmb{g}):=(x,x,\gamma \cdot g_1 \cdot \gamma^{-1},\ldots,\gamma \cdot g_{k-5} \cdot \gamma^{-1},\gamma \cdot h_1\cdot  \gamma^{-1}, \gamma \cdot h_2 \cdot \gamma^{-1}, \gamma \cdot  h_3 \cdot \gamma^{-1})$$
The last three elements are still generators and the pair $(x,\gamma \cdot g_1 \cdot \gamma^{-1})$ shows that it always exists a $1$-reduction.

To conclude the proof, we just showed that the two elements $\pmb{g}$ and $\pmb{g}'$ admit a $1$-reduction to $\bar{\pmb{g}}$ and $\bar{\pmb{g}}'$. By Proposition \ref{prop:liftingandreduction}, $\bar{\pmb{g}}$ and $\bar{\pmb{g}}'$ belong to the same component of $\H(A_6,\cc^{k-1})^{abs}$ and then, by the induction hypothesis, they are equivalent by the Hurwitz action. Proposition \ref{prop:liftinghurwitz} ensures that this action can be lifted to obtain an equivalence also between $\pmb{g}$ and $\pmb{g}'$.
\end{proof}

\section{Curve of genus zero, case of five points} \label{App:MonGenZero}
This section aims to classify the elements of $Ni([G],\cc^5)$ with $G$ being a subgroup of $A_6$. 
We will use the notation $\a{xy--}$ to underline that  we are focusing on a specific part of the permutation and let the other part vary. For example, $\a{1---}$ denotes all the elements with $1$ in the first place. Recall that the external action of an element $s$ of $S_n$ is denoted by $\phi_s$.
There is a natural notion of lexicographic order on $S_n$ that we will use in the calculation. The following lemma ensures that one can always restrict to work with ordered elements.

\begin{lemma} \label{Lemma:ordering}
Up to the Hurwitz action, every element of $Ni([G],\cc^k)$ is equivalent to an ordered one.
\end{lemma}
\begin{proof}
Let $\pmb{g}:=\{g_1, \ldots, g_k\}$ be in $Ni([G],\cc^k)$, if it is not ordered then there is $g_i > g_{i+1}$. Then one can use the Hurwitz action to obtain a new element in which the couple $g_{i+1}, g_{i+1}^{-1} g_i g_{i+1}$ appears. This procedure can be repeated until an ordered element is reached. The process must end due to the fact that the whole number of elements in $Ni([G],\cc^k)$ is bounded.
\end{proof}

Notice that this procedure does not necessarily provide a minimal element of $Ni([G],\cc^k)$. The following element $X$ is ordered but the Hurwitz action on the first two elements produces $Y$ that is still ordered and $Y < X$.
$$X:= \a{1235}, \a{1245}, \a{1326}, \a{1345}, \a{2635}\text{,}$$
$$Y:= \a{1234}, \a{1235}, \a{1326}, \a{1345}, \a{2635}\text{.}$$

The action of the external automorphism can also change the ordering. The following element $X$ is ordered but, by applying $\phi_{\b{56}}$ one finds an element $Y$ that is still ordered and $Y<X$.
$$X:= (\a{1234} , \a{1236} , \a{1345} , \a{1435}, \a{1546})\text{,}$$
$$Y:= (\a{1234} , \a{1235} , \a{1346} , \a{1436}, \a{1645})\text{.}$$


Up to the action of an external automorphism, one can assume the first element to be $\a{1234}$. This assumption reduces the external automorphism that one can use further to be $\phi_{\b{12}}$, $\phi_{\b{34}}$, $\phi_{\b{56}}$ and $\phi_{\a{1324}}$. 
Let now choose an element $\pmb{g}$ in $Ni([G],\cc^5)$. 
The following lemmas apply the Hurwitz action and the external automorphisms in order to find the different classes of these equivalence relations in $Ni([G],\cc^5)$, and then, the different connected components of $\H([G],\cc^5)^{abs}$. Without writing it explicitly, every assumption will be made up to the Hurwitz action and external automorphisms.

\begin{lemma} \label{Lemma:SpecDependingFive}
The element $\pmb{g}=(g_1, \ldots, g_5)$ falls in one of the following cases
\begin{enumerate}
	\item[$\operatorname{[1]}$] $(\a{1234}, \a{1234}, \a{1325}, \a{1346}, \a{2546})$,
	\item[$\operatorname{[2]}$] $(\a{1234}, \a{1324}, \a{1423}, \a{1526}, \a{1526})$,
	\item[$\operatorname{[3]}$] $(\a{1234}, \a{1324}, \a{1425}, \ldots)$,
	\item[$\operatorname{[4]}$] $(\a{1234}, \a{1324}, \a{1456}, \ldots)$,
	\item[$\operatorname{[5]}$] $(\a{1234}, \a{1324}, \a{15--}, \ldots)$.
\end{enumerate}
\end{lemma}
\begin{proof}
Assume $g_1$ to be $\a{1234}$. By Proposition \ref{Prop:reductionfive}, the second element can be either $\a{1234}$, carried on in Part $1$ of the proof, or $\a{1324}$, carried on in Part $2$.

\textbf{Part 1.} The product of the last three components of $\pmb{g}$ must be the identity. For transitivity, there must be at least another $1$ or $2$; up to $\phi_{\b{12}}$, the third element has the form $\a{1abc}$. This gives rise only to three possibilities for the fourth elements. If $g_4$ is $\a{1bac}$ or $\a{1cab}$, we can reduce to Part $2$ of the proof. If $g_4$ is $\a{1axy}$ with $\{x,y\} \neq \{b,c\}$, the element is in the form
$$\pmb{g} = (\a{1234}, \a{1234}, \a{1abc}, \a{1axy}, \a{bcxy})\text{.}$$
The case $a=2$ can not occur: $\pmb{g}$ has to be transitive, then again one need $1$ or $2$ in the last element, but this gives rise to a contradiction. 
If $a=4$ the external automorphism $\phi_{\b{34}}$ allows to reduce to $a=3$, and similarly $a=6$ reduces to $a=5$. It remains to prove that the case $i=5$ reduces to $i=3$. So let now consider
$$\pmb{g} = (\a{1234}, \a{1234}, \a{15bc}, \a{15xy}, \a{bcxy})\text{.}$$
It has to be transitive, then $3$ or $4$ must appear more than one time; up to $\phi_{\b{34}}$, assume $b=3$. The transitivity shows also that $c$ must be different from $4$. The two possibilities for $c$ are only $2$ and $6$. If $c=6$ then $\{x,y\}$ must be equal to $\{2,4\}$. By the Hurwitz action and $\phi_{\b{34}}$ one can reduce to the case $c=2$. This would give the final form 
$$\pmb{g} = (\a{1234}, \a{1234}, \a{1532}, \a{1546}, \a{3246})\text{.}$$
But then, the automorphism $\phi_{\a{1324}}$ gives an equivalence with an element in the case $a=3$. 
$$\pmb{g} = (\a{1234}, \a{1234}, \a{13bc}, \a{13xy}, \a{bcxy})\text{.}$$
Since $\{b,c,x,y\}$ is $\{2,4,5,6\}$, one can assume $b=2$. For transitivity it follows that $c$ is either $5$ or $6$, but then up to $\phi_{\b{56}}$ one gets the final claim for Case (1):
$$\pmb{g} = (\a{1234}, \a{1234}, \a{1325}, \a{1346}, \a{2546})\text{.}$$

\textbf{Part 2.} 
Like before, there must be at least another $1$. Let us first prove that $g_3$ is not in the form $\a{13--}$. In such a case, it would follow that it has to be at least another $1$. So in case there are exactly four $1$, $\pmb{g}$ would be
$$\pmb{g} = (\a{1234}, \a{1324}, \a{13--}, \a{1---}, \a{----})\text{,}$$
but the product being the identity gives easily a contradiction. If follows that the $1$ must be five in total:
$$\pmb{g} = (\a{1234}, \a{1324}, \a{13--}, \a{1---}, \a{1---})\text{.}$$
By some calculations observing that the composition of the last three elements has to be $\a{1423}$, it follows that also this case can not occur.
%
%
%
Due to the external automorphism $\phi_{\b{23}}$ also the case $g_3=\a{12--}$ is not possible. Using external automorphisms it is easy to show that one can always reduce to have $g_3$ equal to $\a{1423}$, $\a{1425}$, $\a{1456}$ or $\a{15--}$, this gives rise to Cases 2, 3, 4, 5.

\textbf{Final form of Case 2.} Let us specify the form of an element in the case $[2]$. Such an element must be in the form
$$\left(\a{1234}, \a{1324}, \a{1423}, \a{abcd}, \a{abcd}\right)$$
And the external automorphisms that can act on it are all the permutation of $\{1,2,3,4\}$ and $\phi_{\b{56}}$. We can then assume $a=1$ and $b=5$ in order to keep the element transitive, it follows than $d$ must be equal to $6$ and, up to external automorphisms one can choose $c$ to be $2$. 
\end{proof}

\begin{lemma} \label{Lemma:List}
The only possible elements that generates the monodromy group, following the simplifications of the previous lemmas are
\begin{enumerate}
\item[$\operatorname{[1]}$] $(\a{1234},\a{1234},\a{1325},\a{1346},\a{2546})$,
\item[$\operatorname{[2]}$] $(\a{1234},\a{1324},\a{1423},\a{1526},\a{1526})$,
\item[$\operatorname{[3.1]}$] $(\a{1234},\a{1324},\a{1425},\a{1623},\a{1635})$,
\item[$\operatorname{[3.2]}$] $(\a{1234},\a{1324},\a{1425},\a{2346},\a{3546})$,
\item[$\operatorname{[4]}$] $(\a{1234},\a{1324},\a{1456},\a{2536},\a{2635})$,
\item[$\operatorname{[5]}$] $(\a{1234},\a{1324},\a{1546},\a{1645},\a{2356})$.
\end{enumerate}
\end{lemma}
\begin{proof}
Code \ref{prog:list5points} lists all the possible ordered elements, belonging to cases $[3]$, $[4]$ and $[5]$ of Lemma \ref{Lemma:SpecDependingFive}. By the Hurwitz action on $g_4$ and $g_5$, the list shrinks to the one presented in the lemma.
\end{proof}

\begin{prop} \label{prop:finalformelements5}
Up to conjugation and external automorphisms every element of $Ni([G],\cc^5)$ with $G$ a subgroup of $A_6$ can be reduced to one of the following cases
\begin{enumerate}
\item[$\operatorname{[1]}$] $(\a{1234},\a{1234},\a{1325},\a{1346},\a{2546})$,
\item[$\operatorname{[2]}$] $(\a{1234},\a{1324},\a{1423},\a{1526},\a{1526})$,
\item[$\operatorname{[3.1]}$] $(\a{1234},\a{1324},\a{1425},\a{1623},\a{1635})$.
\end{enumerate}
and these object are not connected by the action of the braid group because they generate groups of order $60$, $24$ and $360$ respectively. 
\end{prop} 
\begin{proof}
By Lemma \ref{Lemma:List} it remains to prove that $[3.2]$, $[4]$ and $[5]$ can be reduced to one of the three last cases. Cases $[4]$ and $[5]$ are equivalent to Case $[2]$, this can achieved by conjugating $\a{1456},\a{2536}$ in $[4]$ and $\a{1645},\a{2356}$ in $[5]$. Eventually, $[3.2]$ is equivalent to $[3.1]$ by using $\phi_{\b_{14}}$.
\end{proof}

As an immediate consequence of this proposition one obtains the following

\begin{proof}[Proof of Proposition \ref{Prop:FivePoints}]
Each space $\H(A_6,\cc^5)^{abs}$, $\H([G_{24}],\cc^5)^{abs}$ and $\H([G_{60}],\cc^5)^{abs}$ is not empty, thanks to the elements of Cases $[3.1]$, $[2]$ and $[1]$, respectively. The constructions in the lemmas describe explicitly an equivalence between two elements of the same space, giving the connectedness.
\end{proof}

\begin{oss}
In the classification carried on this section we used the external action of the whole group $S_6$, and so we considered only the conjugacy classes of the monodromy groups $[G_{60}]$ and $[G_{24}]$. This is sufficient for the aim of proving Theorems \ref{teor:finalresult} and \ref{teor:finalresultinner}. However, the spaces $\H(G_{60},\cc^5)^{abs}$ and $\H(G_{24},\cc^5)^{abs}$ are also connected for every choice of $G_{60}$ and $G_{24}$ in their conjugacy class. The strategy of proving that consists of proving the connectedness for a particular choice of $G_{24}$ or $G_{60}$, and then using Proposition \ref{prop:InnerAbsolute} to extend the result to all the other cases by conjugation.
\end{oss}

\section{Curve of genus one, case of six points}  \label{App:MonGenOne}
This appendix is devoted to prove Proposition \ref{Prop:FivePoints}. The idea is to exploit Proposition \ref{Prop:reductionfive} taking into account the issues that makes this case different from the induction step.

\begin{ese}
The space $\H(A_6,\cc^6)^{abs}$ has at least two connected components. To see this it is sufficient to compute the lifting invariant of these two elements
$$\a{1234}, \a{1234}, \a{1236}, \a{1236}, \a{1325}, \a{1325}$$
$$\a{1234}, \a{1234}, \a{1236}, \a{1256}, \a{1435}, \a{1456}$$
\end{ese}

One can try to reduce the problem to five points, but the situation here is not easy as for $8$ and more points, since there are not enough elements to use both Propositions \ref{Prop:reductionfive} and \ref{Prop:threeGeneartors}; this is the reason for which one need to treat this case separately and not as a part of the induction step of Theorem \ref{teor:finalresult}.
The strategy is still to use Proposition \ref{Prop:reductionfive} in order to reduce the number of points from $6$ to $5$ and then get rid of the problems. 
\begin{oss} \label{rmk:situations}
Assume that Proposition \ref{Prop:reductionfive} is used on the first five elements of $\pmb{g}$ in $Ni(A_6,\cc^6)$. The following situations can arise:
\begin{enumerate}
	\item A $1$-reduction is possible, and the resulting element is a valid monodromy on $5$ points.
	\item A $1$-reduction is possible, but the transitivity is lost after the reduction.
	\item A $2$-reduction is possible, but in this case, the resulting element would for certain not be transitive, because of the Riemann Hurwitz theorem.
\end{enumerate}
\end{oss}

\begin{lemma}
If $\pmb{g}$ belongs to case $(3)$ of Remark \ref{rmk:situations}, then it is always possible to reduce to case $(1)$ or $(2)$ by using Hurwitz actions and Proposition \ref{Prop:reductionfour}.
\end{lemma}
\begin{proof}
Since $\pmb{g}$ belongs to case $(3)$, we can assume it to be 
$$\{g_1, g_2, g_3, g_4, g_5, g_5\}\text{.}$$
One can apply Proposition \ref{Prop:reductionfour} to the ordered set $\{g_1, g_2, g_3, g_4\}$. If a $1$-reduction is obtained, the proof is concluded. If not, $\pmb{g}$ has one of these forms
$$\{h_1, h_1, h_2, h_2, g_5, g_5\}$$
$$\{\a{1234}, \a{1235}, \a{1634}, \a{1645}, g_5, g_5 \}$$

The first form is given by three pairs of elements ${h_1, h_2, h_3}$ such that $\langle h_1, h_3, g_5 \rangle$ is the whole $A_6$. The second form is simply the third possible outcome of Proposition \ref{Prop:reductionfour}.
Code \ref{prog:solve6pointscase2} shows that every case can be reduced to $(1)$ or $(2)$. The strategy consists in listing all the possibilities for such forms, and then using the Hurwitz action until a $1$-reduction is found.

\end{proof}

The following lemma shows that it is always possible to obtain case $(1)$ in Remark \ref{rmk:situations}.

\begin{lemma}
If $\pmb{g}$ belongs to case $(2)$ of Remark \ref{rmk:situations}, then it is always possible to reduce to case $(1)$ by using Hurwitz actions and Proposition \ref{Prop:reductionfour}.
\end{lemma}
\begin{proof}
This is done with a case by case analysis carried out with Code \ref{prog:list6pointscase1}.
\end{proof}

Then one has only to deal with case $(1)$ of Remark \ref{rmk:situations}.
Proposition \ref{Prop:FivePoints} shows that three case can arise, $(g_1 g_2, g_3, \ldots, g_6)$ belonging to $\H(A_6,\cc^5)^{abs}$, $\H([G_{60}],\cc^5)^{abs}$ or $\H([G_{24}],\cc^5)^{abs}$, respectively.
Let define $\H_-(A_6,\cc^6)^{abs}$ as the space of $(g_1 g_2, g_3, \ldots, g_6)$ that can be reduced to an element in $\H(A_6,\cc^5)^{abs}$ and $\H_+(A_6,\cc^6)^{abs}$ as the space of $(g_1 g_2, g_3, \ldots, g_6)$ that can be reduced to an element in $\H(G_{60},\cc^5)^{abs}$. It remains to study what happens if the element reduces to $\H(G_{24},\cc^5)^{abs}$.

\begin{lemma}
Let $\pmb{g}$ be an element in $Ni(A_6,\cc^5)$ of this form
$$(\a{1234}, \a{1324}, g_3, g_4, g_5, g_6)$$
and assume that the reduction
$$\bar{\pmb{g}} := (\a{1423}, g_3, g_4, g_5, g_6)$$
is a valid monodromy. Then $\bar{\pmb{g}}$ can not have order 24.
\end{lemma}
\begin{proof}
By Proposition \ref{prop:finalformelements5}, the element $\bar{\pmb{g}}$ is equivalent to 
$$(\a{1234}, \a{1324}, \a{1425}, \a{1623}, \a{1635})$$
then the element $\pmb{g}$ equivalent to one of the form
$$(\a{1423}, \a{1324}, \a{1324}, \a{1425}, \a{1623}, \a{1635})$$
with the $1$-reduction taking place in the first element. But the order of this element is not $360$ and this give a contradiction.
\end{proof}

Now it is possible to conclude that $\H(A_6,\cc^6)^{abs}$ has exactly two connected components, proving Proposition \ref{Prop:FivePoints}.

\begin{proof}[Proof of Proposition \ref{Prop:FivePoints}]
From the previous lemmas and propositions, one know that every element of $\H(A_6,\cc^6)^{abs}$ falls either in $\H_+(A_6,\cc^6)^{abs}$, the space of the elements that admits a reduction to an element of $\H([G_60],\cc^5)^{abs}$, or in $\H_-(A_6,\cc^6)^{abs}$, the space of the elements that admits a reduction to an element of $\H(A_6,\cc^5)^{abs}$.
These two spaces are well defined because, as proved in Proposition \ref{prop:liftingandreduction}, the lifting invariant does not change via this kind of reductions and the lifting invariant of $\H(A_6,\cc^5)^{abs}$ and $\H([G_60],\cc^5)^{abs}$ are different.
Eventually, the connectedness of $\H(A_6,\cc^5)^{abs}$ and $\H([G_60],\cc^5)^{abs}$, and Proposition \ref{prop:liftinghurwitz} ensures that $\H_+(A_6,\cc^6)^{abs}$ and $\H_-(A_6,\cc^6)^{abs}$ are also connected.
\end{proof}

\section{The space $\H(A_6,\cc^r)^{in}$}  \label{sec_inner}
If one considers two elements in $\H(A_6,\cc^r)^{in}$, the right action of the external automorphism can be performed only with elements of $A_6$. Proposition \ref{prop:InnerAbsolute} shows that in this setting, two elements with two different lifting invariants of order three are no longer equivalent. The number of different connected components should then increase. Fix a lifting of the identity of order three $\sigma$.
Notice that the expected result is different from the result of \cite{FriedAlternating}, due to the fact that, in that case, the lifting invariant is an order two lifting of the identity.

The following is the proof of the final theorem in the case of inner moduli space.

\begin{proof}[Proof of Theorem \ref{teor:finalresultinner}]
In order to exploit Theorem \ref{teor:finalresult}, notice that an element in $\H_{\conn{1}}(A_6,\cc^k)^{in}$ becomes an element in $\H_{\conn{2}}(A_6,\cc^k)^{in}$ after the right action of a single $2$-cycle. Proposition \ref{prop:InnerAbsolute} shows that there are at least three connected components, that will be denoted by of $\H_{\conn{0}}(A_6,\cc^k)^{in}$, $\H_{\conn{1}}(A_6,\cc^k)^{in}$ and $\H_{\conn{2}}(A_6,\cc^k)^{in}$, depending on the lifting invariant being $\id_V$, $\sigma$ and $\sigma^2$, respectively. It remains to prove that there are no more components. That is straightforward for $\H_{\conn{1}}(A_6,\cc^k)^{in}$ and $\H_{\conn{2}}(A_6,\cc^k)^{in}$, so let us check only the case $\H_{\conn{0}}(A_6,\cc^6)^{in}$.
Consider first the following element in $\H_{\conn{0}}(A_6,\cc^6)^{in}$:
$$\pmb{g} := \a{1234},\a{1324}, \a{1236},\a{1245},\a{1526},\a{4536}.$$
Applying the odd external automorphism $\phi_{\a{1234}\b{56}}$, we obtain the following element
$$\pmb{g}' := \a{1234},\a{1324}, \a{1245},\a{1236},\a{1526},\a{4536}$$
Then by the Hurwitz action we can exchange the third and the fourth elements, going back again to the first element. Hence these two elements in $\H_{\conn{0}}(A_6,\cc^6)^{abs}$ differ from an external automorphism of odd parity and are still related by the Hurwitz action.
Let now $s$ be another element of $S_6$, and consider $\pmb{h} := s^{-1} \pmb{g} s$. If $s$ has even parity, Proposition \ref{prop:InnerAbsolute} show that $\pmb{h} $ still belongs to $\H_{\conn{0}}(A_6,\cc^6)^{in}$, if $s$ has odd parity, then $\pmb{h}$ will differ from $\pmb{g}'$ by the action of $s \circ \phi_{\a{1234}\b{56}}$, that has even parity. Also in this case, then, $\pmb{h}$ still belongs to $\H_{\conn{0}}(A_6,\cc^6)^{in}$.

The same reasoning can be applied to the case of $\H_{\conn{1}}(A_6,\cc^6)^{in}$ and $\H_{\conn{2}}(A_6,\cc^6)^{in}$, and also in the case of $5$ points. Consider for example the following element of $\H_{\conn{1}}(A_6,\cc^6)^{in}$
$$(\a{1234},\a{1234}, \a{1236},\a{1256},\a{1435},\a{1456})$$
Applying the external automorphism $\phi_{\a{1234}\b{56}}$, that is of odd parity, one obtains the following element
$$(\a{1234},\a{1234}, \a{1245},\a{1256},\a{2346},\a{2356})\text{.}$$
that belongs to $\H_{\conn{2}}(A_6,\cc^6)^{in}$. The final part follows exactly as in the case of $\H_{\conn{0}}(A_6,\cc^6)^{in}$. 

Finally, for the induction part to be true it is sufficient to show that the computations of Section \ref{sec_absolute} can be performed with external automorphisms of even parity. These automorphisms are used in Proposition \ref{Prop:FixedPointReduction}, in which an even automorphism is used and Proposition \ref{Prop:reductionfour}, that is used only to show that a particular case does not arise for $k=7$ in the proof of Theorem \ref{teor:finalresult}, but the result is still valid even if one compose with a cycle. Then the proof follows.
\end{proof}

\section{Open problems} \label{sec:open}

The problem of studying the Hurwitz spaces is still widely open. The same technique used in this paper can be in principle used to study all the other cases described asymptotically in \cite{BK}. 
\begin{que}
Complete the study of the lower genus cases described in \cite{BK}, Theorem 4.14, Theorem 4.15 and Proposition 4.16. Find a generalization of the lifting invariant suitable for all the possible conjugacy classes of $A_n$, and use it to classify spaces of mixed monodromy type.
\end{que}
A possible strategy to attach the case of mixed monodromy type is deforming the base making different branch point collide. That in principle could allow us to start from a space with homogeneous monodromy type and then deform it to one with mixed monodromy type. The difficult part would be describing the degenerate situation, in which the monodromy group becomes not transitive. 
\begin{que}
Study from the algebraic point of view what happened if the transitivity hypothesis is dropped in the definition of the Nielsen classes.
\end{que}
Eventually, even by knowing the connected component of the Hurwitz spaces, it is very difficult to explicitly give functions with an assigned type of monodromy. An intriguing question is then the following
\begin{que}
For each component provided by Theorem \ref{teor:finalresult} and Theorem \ref{teor:finalresultinner}, find an explicit example of a rational function defined over $\P^1$ that gives rise to a monodromy belonging to such a component.
\end{que}

\medskip

{\small\noindent{\bf Acknowledgements.}
The first named author was supported by the Department of Mathematics and Natural Sciences of University of Stavanger in the framework of the grant 230986 of the Research Council of Norway. The second named author is partially supported by INdAM (GNSAGA); PRIN 2012 \emph{``Moduli, strutture geometriche e loro applicazioni''} and FAR 2014 (PV) \emph{``Variet\`a algebriche, calcolo algebrico, grafi orientati e topologici''.}
The authors are especially grateful to Professor Michael Fried and Alice Cuzzucoli for their valuable comments on a preliminary version of this paper and to Fedor A. Bogomolov for drawing the work \cite{BK} to our attention. 
}

\appendix
\section{Magma source code} \label{Apx:code}
Most of the codes are just shortcuts in order to list all the elements with a certain property and use the Hurwitz action on them to check if a certain property holds. They are mainly used to work out the basic cases of genus zero and one, and some preliminaries of the induction step that strongly depends on the fact that we are restricting ourselves to work in the conjugacy class $\cc$. 

\subsection{Preamble to the other codes} \label{prog:preamble}
This code is the preamble to all the other codes. 
\begin{codice_magma}
S:= PermutationGroup<6 |  (1, 2)(3,4),(1,2,4,5)(3,6)>;

// The ordered list of the elements in the conjugacy class of the product of two disjoint cycles in A_6.
El:=[S!(1,2)(3,4),S!(1,2)(3,5),S!(1,2)(3,6),S!(1,2)(4,5),S!(1,2)(4,6),S!(1,2)(5,6),S!(1,3)(2,4),S!(1,3)(2,5),
     S!(1,3)(2,6),S!(1,3)(4,5),S!(1,3)(4,6),S!(1,3)(5,6),S!(1,4)(2,3),S!(1,4)(2,5),S!(1,4)(2,6),S!(1,4)(3,5),
     S!(1,4)(3,6),S!(1,4)(5,6),S!(1,5)(2,3),S!(1,5)(2,4),S!(1,5)(2,6),S!(1,5)(3,4),S!(1,5)(3,6),S!(1,5)(4,6),
     S!(1,6)(2,3),S!(1,6)(2,4),S!(1,6)(2,5),S!(1,6)(3,4),S!(1,6)(3,5),S!(1,6)(4,5),S!(2,3)(4,5),S!(2,3)(4,6),
     S!(2,3)(5,6),S!(2,4)(3,5),S!(2,4)(3,6),S!(2,4)(5,6),S!(2,5)(3,4),S!(2,5)(3,6),S!(2,5)(4,6),S!(2,6)(3,4),
     S!(2,6)(3,5),S!(2,6)(4,5),S!(3,4)(5,6),S!(3,5)(4,6),S!(3,6)(4,5)];

// The following function uses the Hurwitz action on a Nielsen class G and check by using the function
// 'FunctionCheck' that a certain property holds. FunctionCheck takes as input a Nielsen class
// and returns true or false depending on if that property holds.
// N is the number of times the generators of the Braid group are applied. To speed up the computation
// often it is just needed to apply them one or two times
HurwitzAction := function(G,N,FunctionCheck)
    if (FunctionCheck(G)) then return true; end if;
    Result:={G};
    for P:=1 to N do
    for I in Result do
        for J:= 1 to #I-1 do
            NewElement:=I;
            A:=NewElement[J];
            B:=NewElement[J+1];
            NewElement[J]:=A*B*Inverse(A);
            NewElement[J+1]:=A;
            if (FunctionCheck(NewElement)) then return true; end if;
            Result := Result join {NewElement};
        end for;
    end for;
    end for;
    return false;
end function;
\end{codice_magma}
\vspace{-2mm}

\subsection{Lifting invariant} \label{prog:liftinginvariant}\hfill \break
This code computes the lifting invariant of an element in $\Ni(A_6,\cc^k)$.
\begin{codice_magma}
// The valentiner group H
H:=PermutationGroup<18|(2, 6)(4, 11)(7, 9)(8, 13)(10, 14)(12, 16),
   (1, 2, 7, 4)(3, 8, 6, 10)(5, 9, 13, 12)(11, 15)(14, 17)(16, 18)>;
LiftingHom:=hom< H -> S | H.1 -> S.1, H.2 -> S.2 >;
ElKer:=H!(1, 3, 5)(2, 8, 9)(4, 10, 12)(6, 13, 7)(11, 14, 16)(15, 17, 18);

LiftingInvariant := function(G)
    ProductH:=Identity(H);
    for El in G do
        Lift:=El@@LiftingHom;
        while (Order(Lift) ne Order(El)) do
            Lift := Lift * ElKer;
        end while;
        ProductH:=ProductH*Lift;
    end for;
    return ProductH;
end function;
\end{codice_magma}
\vspace{-2mm}

\subsection{Finding minimal subsets of generators I} \label{prog:from5to4} \hfill \break
This code allows to prove that every set of $5$ elements in $\cc$ that generates all $A_6$ admits a subset of $4$ elements that are still generators. 
\begin{codice_magma}  
// This function checks if one can remove an element from the list G 
// and still get something of maximal order
ReduceByOne := function(G)
    for R in G do
	NewElements:=Exclude(G,R);
	if (Order(sub<S|NewElements>) eq 360) then
		return true;
	end if;
    end for;
    return false;
end function;

CountNonReducibles:=0;
for I in [2 .. 45], J in [I+1 .. 45], K in [J+1 .. 45], L in [K+1 .. 45] do
	Candidate:=[S!(1,2)(3,4),El[I],El[J],El[K],El[L]];
	if (Order(sub<S|Candidate>) eq 360 and not ReduceByOne(Candidate)) then
            CountNonReducibles:=CountNonReducibles+1;
	end if;
end for;
print "Number of non reducible elements: ",CountNonReducibles;
\end{codice_magma}
\vspace{-2mm}

\subsection{Finding minimal subsets of generators II} \label{prog:final4to3} \hfill \break
This code allows to prove that in every set of $4$ elements in $\cc$ that generates all $A_6$ one can find $3$ elements that are still generators up to the Hurwitz action on the set.

\begin{codice_magma} 
// This function checks if one can remove an element from the list G 
// and still get something of maximal order
ReduceByOne := function(G)
    for R in G do
	NewElements:=Exclude(G,R);
	if (Order(sub<S|NewElements>) eq 360) then
		return true;
	end if;
    end for;
    return false;
end function;

CountNonReducibles:=0;
for I in [2 .. 45], J in [I+1 .. 45], K in [J+1 .. 45] do
	Candidate:=[S!(1,2)(3,4),El[I],El[J],El[K]];
	if (Order(sub<S|Candidate>) eq 360 and not ReduceByOne(Candidate)) then
            if (not HurwitzAction(Candidate,2,ReduceByOne)) then
                CountNonReducibles:=CountNonReducibles+1;
            end if;
	end if;
end for;
print "Number of non reducible elements: ",CountNonReducibles;
\end{codice_magma}
\vspace{-2mm}

\subsection{Finding reductions I} \label{prog:find2fixedpoints} \hfill \break
This codes try to list the sets of $4$ elements of $\cc$ for which it is not possible to find any pair of elements with the same fixed points, even if the Hurwitz action is used.
\begin{codice_magma} 
// This function checks if among the elements of the list G
// there are two that have the same fixed points
IsReducible := function(G)
    for R in [1..(#G-1)], K in [R+1..#G] do
        if (#(Fix(G[R]) meet Fix(G[K])) eq 2) then
            return true;
        end if;
    end for;
    return false;
end function;

for I in [2 .. 45], J in [I+1 .. 45], K in [J+1 .. 45] do
    	Candidate:=[S!(1,2)(3,4),El[I],El[J],El[K]];
        if (not HurwitzAction(Candidate,10,IsReducible)) then
            printf "
        end if;
end for;
\end{codice_magma}
\vspace{-2mm}

\subsection{Finding reductions II} \label{prog:find2fixedpoints5} \hfill \break
This code allows to prove that in every set of $5$ elements in $\cc$ it is possible to find two elements with the same fixed points up to the Hurwitz action.

\begin{codice_magma}
// This function checks if among the elements of the list G
// there are two that have the same fixed points
IsReducible := function(G)
    for R in [1..(#G-1)], K in [R+1..#G] do
        if (#(Fix(G[R]) meet Fix(G[K])) eq 2) then
            return true;
        end if;
    end for;
    return false;
end function;

CountNonReducibles:=0;
for I in [2 .. 45], J in [I+1 .. 45], K in [J+1 .. 45], N in [K+1 .. 45] do
    	Candidate:=[S!(1,2)(3,4),El[I],El[J],El[K],El[N]];
        if (not HurwitzAction(Candidate,10,IsReducible)) then
            CountNonReducibles:=CountNonReducibles+1;
        end if;
end for;
print "Number of non reducible elements: ",CountNonReducibles;
\end{codice_magma}
\vspace{-2mm}

\subsection{Special case of five points} \label{prog:list5points} \hfill \break
This code provides the list of elements in $Ni(G,\cc^5)$ requested by Lemma \ref{Lemma:List}.

\begin{codice_magma}
// All the possible cases for the third element given by the lemma
Base3:=[S!(1,4)(2,5),S!(1,4)(5,6),S!(1,5)(2,3),S!(1,5)(2,4),S!(1,5)(2,6),
                       S!(1,5)(3,4),S!(1,5)(3,6),S!(1,5)(4,6)];

for I in [1 .. 8], J in [1 .. 45], K in [J .. 45] do
    Candidate:=[S!(1,2)(3,4),S!(1,3)(2,4),Base3[I],El[J],El[K]];
    if (&*Candidate eq Identity(S) and IsTransitive(sub<S|Candidate>)) then
        printf "
    end if;
end for;
\end{codice_magma}
\vspace{-2mm}

\subsection{Special case of six points I} \label{prog:solve6pointscase2} \hfill \break
This code lists all the possible cases of an element in $\Ni(A_6,\cc^6)$ having only $2$-reductions and use the Hurwitz action to prove that actually these elements also admit a $1$-reduction.
\begin{codice_magma}
// This function checks if among the elements of the list G 
// there are two that gives a 1-reduction
IsOneReducible:= function(G)
for R in [1..(#G-1)], K in [R+1..#G] do
    if ((Order(sub<S|[G[R],G[K]]>) eq 4)  and (#(Fix(G[R]) meet Fix(G[K])) eq 2)) then
        return true;
    end if;
end for;
    return false;
end function;

CountNonReducibles:=0;
for I in [2 .. 30], J in [I .. 45] do
    Candidate:=[S!(1,2)(3,4),S!(1,2)(3,4),El[I],El[I],El[J],El[J]];
    if (Order(sub<S|Candidate>) eq 360 and not HurwitzAction(Candidate,5,IsOneReducible)) then
        CountNonReducibles:=CountNonReducibles+1;
    end if;
 end for;

for I:= 1 to 45 do
    Candidate:=[S!(1,2)(3,4),S!(1,2)(3,5),S!(1,6)(3,4),S!(1,6)(4,5),El[I],El[I]];
    if (Order(sub<S|Candidate>) eq 360 and not HurwitzAction(Candidate,5,IsOneReducible)) then
        CountNonReducibles:=CountNonReducibles+1;
    end if;
end for;

print "Number of non reducible elements: ",CountNonReducibles;
\end{codice_magma}
\vspace{-2mm}

\subsection{Special case of six points II} \label{prog:list6pointscase1} \hfill \break
This code lists all the possible cases of an element in $\Ni(A_6,\cc^6)$ that has a $1$-reduction but does not give an element in $\Ni(A_6,\cc^5)$ because the transitivity is lost.
\begin{codice_magma}
// This function checks if by multiplying the first two elements of the list G 
// one still get something that generates a transitive subgroup
IsReducible:= function(G)
    G2:=[G[1]*G[2],G[3],G[4],G[5],G[6]];
    if IsTransitive(sub<S|G2>) then return true; end if;
    return false;
end function;

CountNonReducibles:=0;
for I in [1 .. 45], J in [I .. 45], K in [J .. 45], N in [K .. 45] do
    Candidate:=[S!(1,2)(3,4),S!(1,3)(2,4),El[I],El[J],El[K],El[N]];
    Candidate2:=[S!(1,2)(3,4)*S!(1,3)(2,4),El[I],El[J],El[K],El[N]];
    Product:=&*Candidate;
    G:=sub<S|Candidate>;
    G2:=sub<S|Candidate2>;
    if (Product eq Identity(S) and Order(G) eq 360 and not IsTransitive(G2)
    							and not HurwitzAction(Candidate,1,IsReducible)) then
        CountNonReducibles:=CountNonReducibles+1;
    end if;
end for;

print "Number of non reducible elements: ",CountNonReducibles;
\end{codice_magma}

\end{document}